\pdfoutput 1
\documentclass{article} \usepackage{amssymb}
\usepackage{amsfonts} \usepackage{amsmath} \usepackage{amsthm} \usepackage{bm}
\usepackage{latexsym} \usepackage{epsfig}
\usepackage{subfig}
\usepackage{graphicx}

\newtheorem*{theorem*}{Theorem}
\newtheorem{theorem}{Theorem}[section]
\newtheorem{lemma}[theorem]{Lemma}
\newtheorem*{proposition*}{Proposition}
\newtheorem*{corollary*}{Corollary}

\newtheorem{corollary}[theorem]{Corollary}

\newtheorem{definition}[theorem]{Definition}

\newcommand{\ignore}[1]{}

\newcommand{\enote}[1]{} 
\newcommand{\knote}[1]{}
\newcommand{\rnote}[1]{}

\renewcommand{\P}[1]{{\mathbb{P}}\left[{#1}\right]}

\newcommand{\E}[1]{{\mathbb{E}}\left[{#1}\right]}

\newcommand{\N}{\mathbb N} \newcommand{\R}{\mathbb R}
\newcommand{\Z}{\mathbb Z} 
\newcommand{\C}{\mathbb C}
\newcommand{\Q}{\mathbb Q}

\newcommand{\half}{{\textstyle \frac12}}

\newcommand{\rw}{\gamma}

\begin{document}

\title{Scenery Reconstruction on Finite Abelian Groups}

\author{Hilary Finucane\footnote{Supported by an ERC grant.}, Omer
  Tamuz\footnote{Supported by ISF grant 1300/08. Omer Tamuz is a
    recipient of the Google Europe Fellowship in Social Computing, and
    this research is supported in part by this Google Fellowship.} and
  Yariv Yaari\footnote{Weizmann Institute, Rehovot 76100, Israel}}

\maketitle

 \begin{abstract}

   We consider the question of when a random walk on a finite abelian
   group with a given step distribution can be used to reconstruct a
   binary labeling of the elements of the group, up to a
   shift. Matzinger and Lember (2006) give a sufficient condition for
   reconstructibility on cycles. While, as we show, this condition is
   not in general necessary, our main result is that it is necessary
   when the length of the cycle is prime and larger than 5, and the
   step distribution has only rational probabilities. We extend this
   result to other abelian groups.

%  A binary labeling of a graph is a function from its nodes to
%  $\{0,1\}$.  Scenery reconstruction is the problem of inferring a
%  labeling given the labels observed by a particle performing a random
%  walk on the graph.

%  We consider the question of when a random walk on a finite abelian
%  group with a given step distribution is {\em reconstructive}, or can
%  be used to reconstruct a binary labeling up to a shift. We focus
%  initially on walks on undirected cycles. Matzinger and Lember (2006)
%  give a sufficient condition for reconstructibility on cycles,
%  involving the Fourier transform of the random walk's step
%  distribution. While, as we show, this condition is not in general
%  necessary, our main result is that it is necessary when the length
%  of the cycle is prime and larger than 5, and the step distribution
%  has only rational probabilities.

%  We use this result to show that rational random walks on prime
%  cycles which have non-zero drift (for an appropriate notion of
%  drift) are reconstructive, as is any random walk with a
%  non-symmetric bounded step function, on a large enough cycle.
  
%  We extend our results to walks on a large class of abelian groups,
%  namely products of prime cycles. These include some cycles of
%  composite length as well as regular tori of prime length.
\end{abstract}

\section{Introduction}
Benjamini and Kesten~\cite{BenjKesten:96} consider the following
model: Let $G=(V,E)$ be a graph and let $f_1,f_2:V \to \{0,1\}$ be
binary labelings of the vertices, or ``sceneries". Let $v(t)$, for $t
\in \N$, be the position of a particle performing a random walk on
$G$. Given an observation of one of the sequences
$\{f_1(v(t))\}$ or $\{f_2(v(t))\}$, is it
possible to decide which of the two sequences was observed? This is
the problem of distinguishing sceneries. Benjamini and Kesten give
some conditions under which one can distinguish correctly with
probability one, and show that in some cases this cannot be done. Lindenstrauss~\cite{Lindenstrauss:1999} showed that when $G=\Z$ there
exist uncountably many functions $f$ that cannot be distinguished
given $\{f(v(t))\}$. 

The problem of scenery reconstruction is that of learning a completely
unknown scenery $f:V \to \{0,1\}$ by observing $\{f(v(t))\}$.
L\"owe, Matzinger and Merkel~\cite{lowe2004reconstructing} showed that when $G=\Z$ and the
values of $f$ are chosen i.i.d.\ uniformly (from a large enough set),
then almost all functions $f$ can be reconstructed. Furthermore,
Matzinger and Rolles~\cite{Matzinger:2003} showed that $f$ can be
reconstructed in the interval $[-n,n]$ with high probability from a
polynomial sample. Further related work has been pursued by
Howard~\cite{Howard:96, howard1996orthogonality,howard1997distinguishing}, and a good overview is given by Kesten~\cite{kesten1998distinguishing}.

We focus on the case when the graph $G$ is an undirected cycle of size $n$. One may think of this graph as having $\Z_n = \Z/n\Z$ as its vertex set, with $(k,\ell) \in E$ whenever $k-\ell \in \{-1,1\}$; equivalently, $G$ is the Cayley graph of $\Z_n$ with generating set $\{-1,1\}$.  We characterize a random walk on this graph by a step distribution $\rw$ on $\Z_n$ such that at each turn with probability $\rw(k)$ the particle jumps $k$ steps: $\rw(k) = \P{v(t+1)-v(t)=k}$. Indeed, this is simply a random walk on the group $\Z_n$. We choose $v(1)$ uniformly from $\Z_n$.

%We extend our results to a more general class of finite abelian
%groups, for which the geometry of the graph $G$ is that of a torus of
%some dimension.

%In this setting we explore the conditions under which $f$ can be reconstructed from $\{f(v(t))\}$. (Note that for these cyclic sceneries, reconstruction is equivalent to distinguishing between any two pairs of sceneries that differ by more than a shift.) 
Let the functions $f_1:\Z_n \to \{0,1\}$ and $f_2:\Z_n \to \{0,1\}$ be
two sceneries. Fixing the random walk $\rw$, we say that $f_1$ and
$f_2$ belong to the same equivalence class if the distribution of
$\{f_1(v(t))\}$ is equal to $\{f_2(v(t))\}$. Note that for the finite
case, unlike the infinite case, these two distributions are different
if and only if they are mutually singular, if and only if $f_1$ and
$f_2$ can be distinguished with probability one from a single
realization of $\{f(v(t))\}$. Thus, our equivalence relation can also
be thought of as indistinguishability.

A first observation is that if $f_1$ and $f_2$ differ only by a cyclic
shift, i.e.\ there exists an $\ell$ such that for all $k$ it holds that
$f_1(k)=f_2(k+\ell)$ (again using addition in $\Z_n$), then $f_1$ and
$f_2$ are in the same equivalence class. Hence the equivalence classes
of functions that cannot be distinguished contain, at the least, all
the cyclic shifts of their members.

The trivial random walk that jumps a single step to the right
w.p.\ $1$ clearly induces minimal equivalence classes; that is,
classes of functions related by cyclic shifts and nothing
more. Another trivial random walk - the one that jumps to any $k \in
\Z_n$ with uniform probability - induces equivalence classes that
contain all the functions with a given number of ones.

For some random walks the classes of indistinguishable functions can
be surprising. Consider the following example, illustrated in
Fig.~\ref{fig:hilary-walk} below. The random walk is on the cycle of
length $n$, for $n$ divisible by $6$ ($n=12$ in the figure). Its step
function is uniform over $\{-2,-1,1,2\}$, so that it jumps either two
steps to the left, one to the left, one to the right or two to the
right, all with probability $1/4$.

Let $f_1(k)$ be $0$ for even $k$ and $1$ otherwise. Let $f_2(k)$ be
$0$ for $k \mod 6 \in \{0,1,2\}$ and $1$ otherwise. It is easy to
see that these two functions are indistinguishable, since
the sequence of observed labels will be a sequence of uniform i.i.d.\
bits in both cases.

\begin{figure}[h]
  \subfloat[$f_1$]
  {
    \label{fig:f_1}
    \includegraphics[width=0.3\textwidth]{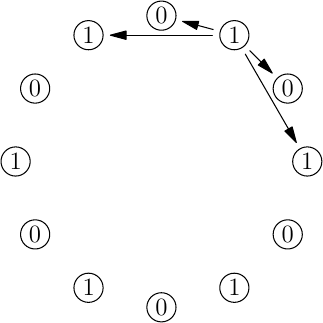}
  }
  \hspace{80pt}
  \subfloat[$f_2$]
  {
    \label{fig:f_2}
    \includegraphics[width=0.3\textwidth]{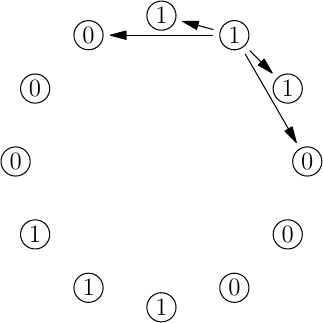}
  }
  \centering
  \caption{\label{fig:hilary-walk} The random walk depicted here has a
    step distribution that is uniform over $\{-2,-1,1,2\}$. It cannot be
    used to distinguish the two very different sceneries $f_1$ and
    $f_2$ above.}
\end{figure}

%Note that in the problem of scenery reconstruction on infinite groups, it may be the case that the distributions of $\{f_1(v(t))\}$ and $\{f_2(v(t))\}$ are different (allowing them to be distinguished) but not mutually singular (allowing reconstruction). However, in the finite case (equivalently, the periodic case), if the distributions are different for every different $f_1$ and $f_2$ then they are also mutually singular for every different $f_1$ and $f_2$. It follows that in this case one can reconstruct from a single sample of $\{f(v(t))\}$.

The question that we tackle is the following: which random walks induce minimal equivalence classes? In other words, for which random walks can any two sceneries that differ by more than a shift be distinguished? In the finite case, when any two sceneries differing by more than a shift can be distinguished, it is also possible to reconstruct any scenery up to a shift. So we call a random walk that induces minimal equivalence classes {\bf reconstructive}:
\begin{definition}
  Let $\rw:H \to \R$ be the step distribution of a random walk $v(t)$
  on a finite group $H$, so that $v(1)$ is picked uniformly from $H$
  and $\rw(k) = \P{v(t+1)-v(t) = k}$. Then $v(t)$ is {\bf
    reconstructive} if the distributions of
  $\{f_1(v(t))\}_{t=1}^\infty$ and $\{f_2(v(t))\}_{t=1}^\infty$ are
  identical only if $f_1$ is a shift of $f_2$.
\end{definition}

We are interested in exploring the conditions under which $v(t)$ is
reconstructive. Howard~\cite{Howard:96} answers this question for
$\rw$ with support on $\{-1,0,1\}$. For symmetric walks in which
$\rw(-1) = \rw(1) \neq 0$ he shows that $f$ can be reconstructed up to
a shift and a mirror image flip (that is, $f_1$ and $f_2$ cannot be
distinguished when $f_1(k)=f_2(-k)$). In all other cases (except the
trivial $\rw(0)=1$) he shows that the equivalence classes are minimal.

Matzinger and Lember~\cite{matzinger2006reconstruction} introduce the
use of the Fourier transform to the study of this question.  They
prove the following theorem\footnote{Matzinger and Lember's notion of
  a reconstructive random walk is slightly different, in that they
  only require reconstruction up to a shift {\em and flip}, where the
  {\em flip} of $f(k)$ is $f(-k)$. Thus Theorem~\ref{thm:distinct-f}
  differs in this point from theirs. They also require reconstruction
  from a single sequence of observations, which, as we point out
  above, is equivalent to our notion in the case of finite groups.} :
\begin{theorem}[Matzinger and
  Lember~\cite{matzinger2006reconstruction}, Theorem 3.2]
  \label{thm:distinct-f}
  Let $\rw$ be the step distribution of a random walk $v(t)$ on $\Z_n$. Let
  $\hat{\rw}$ be the Fourier Transform of $\rw$. Then $v(t)$ is
  reconstructive if the Fourier coefficients $\{\hat{\rw}(x)\}_{x \in
    \Z_n}$ are distinct.
\end{theorem}
We provide the proof for this theorem in the appendix for the reader's
convenience, extending it (straightforwardly) to random walks on any
abelian group\footnote{Matzinger and Lember state it
  for ``periodic sceneries on $\Z$'', which are equivalent to
  sceneries on cycles.}. We henceforth use $\hat{\rw}$ to denote the
Fourier transform of $\rw$.

This condition is not necessary, as we show in
Theorem~\ref{thm:counterexample}. Our main result is that this
condition {\em is} necessary for random walks on $\Z_n$ when $n$ is
prime and larger than 5, and $\rw$ is rational.
\begin{theorem}
  \label{thm:prime-distinct}
  Let $\rw$ be the step distribution of a random walk $v(t)$ on $\Z_n$, for
  $n$ prime and larger than five, and let $\rw(k)$ be rational for all
  $k$.  Then $v(t)$ is reconstructive only if the Fourier
  coefficients $\{\hat{\rw}(x)\}_{x \in \Z_n}$ are distinct.
\end{theorem}

When a step distribution $\rw$ is rational (i.e., $\rw(k)$ is rational
for all $k$) then it can be viewed as a uniform distribution over a
finite multiset $\Gamma$. Here a finite multiset is a finite
collection of elements with repetitions.

Using this representation we make the following definition:
\begin{definition}
  \label{def:drift}
  The {\em drift} of a random walk on $\Z_n$ with step distribution
  uniform over the multiset $\Gamma$ is $D(\Gamma) = \sum_{k \in
    \Gamma} k$ (with addition in $\Z_n$).
\end{definition}
We show that a walk with non-zero drift is reconstructive on
prime-length cycles.
\begin{theorem}
  \label{cor:drift}
  Suppose $n$ is prime and greater than 5, and suppose
  $v(t)$ is a random walk over $\Z_n$, with step distribution uniform
  over the multiset $\Gamma$. Then if $D(\Gamma) \neq 0$ then $v(t)$
  is reconstructive.
\end{theorem}

We show that any fixed, bounded rational random walk on a large enough
prime cycle is either symmetric or reconstructive. Symmetric random
walks are those for which $\rw(k)=\rw(-k)$ for all $k$. They are not
reconstructive because they cannot distinguish between $f(k)$ and
$f(-k)$.

Given a step distribution $\rw : \mathbb{Z} \rightarrow [0,1]$ on
$\Z$, let $\rw_n$ denote the step distribution on $\mathbb{Z}_n$
induced by $\rw$ via
$$\rw_n(k)= \sum_{a = k\, mod\, n} \rw(a).$$
\begin{theorem}\label{cor:bounded}
  Let $\rw : \Z \rightarrow [0,1]$ be a distribution over $\Z$ with
  bounded support, and assume $\rw(a) \in \mathbb{Q}$ for all
  $a$. Then either $\rw$ is symmetric, or there exists an $N$ such
  that for all prime $n > N$, $\rw_n$ is reconstructive.
\end{theorem}

Finally, we extend the result of Theorem~\ref{thm:prime-distinct} to
random walks on any abelian group of the form $Z_{p_1}^{d_1} \times \cdots \times
Z_{p_m}^{d_m}$ where $p_1, \ldots, p_m$ are primes larger than $5$.
\begin{theorem}\label{thm:general}
  Let $\rw$ be the step distribution of a random walk $v(t)$ on
  $\mathbb{Z}_{p_1}^{d_1} \times \cdots \times
  \mathbb{Z}_{p_m}^{d_m}$, such that $\rw(k_1, \ldots, k_m) \in \Q$
  for all $(k_1, \ldots , k_m)$, and suppose that $p_i>5$ is prime for
  all $i$. Then $v(t)$ is reconstructive only if the Fourier
  coefficients $\{\hat{\rw}(x)\}_{x \in \Z_n}$ are distinct.
\end{theorem}
This result applies, in particular, to walks on $\Z_n$ where $n$ is
square-free (i.e., is not divisible by a square of a prime) and not
divisible by $2$, $3$ or $5$, since such groups are isomorphic to
groups of the form considered in this theorem.

We show that the rationality condition of our main theorem is tight by
presenting an example of a random walk with prime $n$, irrational
probabilities and non-distinct Fourier coefficients where
reconstruction is possible. It remains open, however, whether the
result holds for rational random walks on general cycles of composite
length.

This paper will proceed as follows. In
Section~\ref{sec:preliminaries}, we will review some useful algebraic
facts.  In Section~\ref{sec:mainproof}, we will prove
Theorem~\ref{thm:prime-distinct}. In Section~\ref{sec:corollaries} we
will prove Theorem~\ref{cor:drift}, on walks with non-zero drift, and
Theorem~\ref{cor:bounded}, on walks with bounded support. In
Sections~\ref{sec:tightness}, and \ref{sec:extensions} we show the
tightness of the main theorem, and extend it to products of
prime-length cycles, respectively. In Section~\ref{sec:openproblems},
we will present some open problems.

\section{Preliminaries}
\label{sec:preliminaries}

Before presenting the proofs of our theorems, we would like to refresh
the reader's memory of abelian (commutative) groups and their Fourier
transforms, as well as introduce our notation. For the cyclic group
$\Z_n$, the Fourier transform $\hat{\rw}: \Z_n \to \C$ of $\rw :\Z_n
\to \C$ is defined by
\begin{align*}
  \hat{\rw}(x) = \sum_{k \in \Z_n}\omega_n^{kx}\rw(k). 
\end{align*}
where $\omega_n = e^{-\frac{2\pi}{n}i}$.

Any finite abelian group can be written as the Cartesian product of
cycles with prime power lengths: if $H$ is abelian then there exist
$\{n_i\}$ such that $H$ is isomorphic to the {\em torus}
$\Z_{n_1}\times \cdots \times \Z_{n_m}$.  Thus an element $h \in H$
can be thought of as a vector so that $h_i$ is in $\Z_{n_i}$.

When all the $n_i$'s are the same then $H=\Z_n^d$ is a {\em regular
  torus}, and when $n$ is prime then it is a {\em prime regular
  torus}. Here a natural dot product exists for $h,k \in H$: $h \cdot
k=\sum_{i=1}^dh_ik_i$, where both the multiplication and summation are
over $\Z_n$. The Fourier transform for $\Z_n^d$ is thus
\begin{align*}
  \hat{\rw}(x) = \sum_{k \in \Z_n^d}\omega_n^{k \cdot x}\rw(k). 
\end{align*}

The representation of abelian groups as tori is not unique. The
canonical representation is $H=\Z_{n_1}^{d_1}\cdots\Z_{n_m}^{d_m}$,
where the $n_i$'s are distinct powers of primes. In this general case
an element $h \in H$ can be thought of as a vector of vectors, so that
$h_i$ is in $\Z_{n_i}^{d_i}$. The Fourier transform is
\begin{align*}
  \hat{\rw}(x) = \sum_{k \in H}\prod_{i=1}^m\omega_{n_i}^{k_i \cdot x_i}\rw(k). 
\end{align*}

The {\em $n$-th cyclotomic polynomial} is the minimal polynomial of
$\omega_n$; i.e. the lowest degree polynomial over $\Q$ that has
$\omega_n$ as a root and a leading coefficient of one. This polynomial
must divide all non-zero polynomials over $\Q$ that have $\omega_n$ as
a root. When $n$ is prime then this polynomial is $Q_n(t)=
\sum_{i=0}^{n-1} t^i$.

\section{Proof of Theorem~\ref{thm:prime-distinct}}
\label{sec:mainproof}

\subsection{Outline of main proof}
Theorem~\ref{thm:prime-distinct} states the distinctness of the Fourier
coefficients of the step distribution $\rw$ is a necessary (and by
Theorem~\ref{thm:distinct-f} sufficient) condition for reconstruction,
when $n$ is prime and greater than five, and the probabilities
$\rw(k)$ are rational. In Theorem~\ref{thm:counterexample} we give a
counterexample that shows that the rationality condition is tight,
giving an irrational random walk on $\Z_7$ with non-distinct Fourier
coefficient, which is reconstructive. In Section~\ref{sec:tightness}, we
show that the condition $n>5$ is also tight.

To prove Theorem~\ref{thm:prime-distinct} we shall construct,
for any random walk on a cycle of prime length larger than 5 with rational step distribution
$\rw$ such that $\hat{\rw}(x) = \hat{\rw}(y)$ for some $x \neq y$, two functions $f_1$ and $f_2$ that differ by more than a
shift. We shall consider two random walks: $v_1(t)$ on $f_1$ and
$v_2(t)$ on $f_2$, with the same step distribution $\rw$. If a
coupling exists such that the observed sequences $\{f_1(v_1(t))\}$ and
$\{f_2(v_2(t))\}$ are identical then $f_1$ and $f_2$ cannot be
distinguished, since the sequences will have the same marginal
distribution. We will show that such a coupling always exists.

\subsection{Motivating example}
  
Consider the following example. Let the step distribution $\rw$ uniformly
choose to move 1, 2, or 4 steps to the left on a cycle of length 7. A
simple calculation will show that $\hat{\rw}(1) = \hat{\rw}(2)$. In
what is not a coincidence (as we show below), the support of $\rw$ is
invariant under multiplication by 2 (over $Z_7$):  $2\{1,2,4\} =
\{1,2,4\}$. We will use this fact to construct two functions $f_1$ and
$f_2$ that are indistinguishable, even though they are not related by
a shift.

Let $f_1(k)$ equal 1 for $k \in \{0,1\}$ and $f_2(k)$ equal 1 for $k
\in \{0,2\}$, as shown in Fig.~\ref{fig:seven-walk}. (More generally, $f_1$ and $f_2$ can be any two functions such that $f_1(k) = f_2(2k)$ for all $k$.) Let $v_1(t)$ be a
random walk over $Z_7$ with step distribution $\rw$. We couple it to a
random walk $v_2(t)$ by setting $v_2(t) = 2 \cdot v_1(t)$, so that
$v_2$ always jumps twice as far as $v_1$. Its step distribution is
thus uniform over $2 \cdot \{1,2,4\}=\{1,2,4\}$, and indeed has the
same step distribution $\rw$. Finally, because $f_2(2k) = f_1(k)$, we
have $f_2(v_2(t)) = f_1(v_1(t))$ for all $t$.

We have thus constructed two random walks, both with step distribution
$\rw$. The first walks on $f_1$ and the second walks on $f_2$, and yet
their observations are identically distributed. Hence the two
functions are indistinguishable.

Note also that the sum of the elements in the support of $\rw$ is
zero: $1+2+4=0 \mod 7$. This is not a coincidence; as we show
in Theorem~\ref{cor:drift} this must be the case for every step
distribution that is not reconstructive.

\begin{figure}[h]
  \subfloat[$f_1$]
  {
    \label{fig:f_1}
    \includegraphics[width=0.3\textwidth]{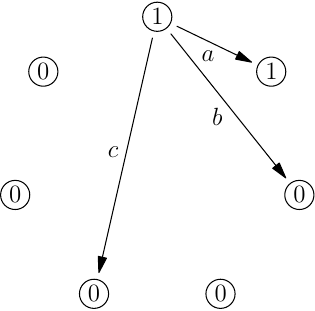}
  }
  \hspace{80pt}
  \subfloat[$f_2$]
  {
    \label{fig:f_2}
    \includegraphics[width=0.3\textwidth]{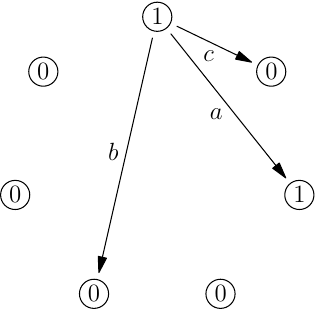}
  }
  \centering
  \caption{\label{fig:seven-walk} We couple the random walks on $f_1$
    and $f_2$ above by choosing, at each time period, either `a', `b'
    or `c' uniformly and having each walk take the step marked by that
    letter in its diagram. The result is that $f_1(v_1(t)) =
    f_2(v_2(t))$. }
\end{figure}

\subsection{Proof of Theorem~\ref{thm:prime-distinct}}
To prove Theorem~\ref{thm:prime-distinct} we will first prove two
lemmas. We assume here that $\rw(k)$ is rational for all $k$ and that
$n>5$ is prime.

\begin{lemma}\label{lemma:xS=yS} Suppose that $\hat{\gamma}(x) =
  \hat{\gamma}(y)$ for some $x, y \neq 0$. Then $\gamma(k
  x^{-1} y) = \gamma(k)$ for all $k$, where all operations are in the field $\mathbb{Z}_n$.
\end{lemma}
In other words, the random walk has the same probability to add $k$ or
$kx^{-1} y$, for all $k$.

\begin{proof}

  Letting $\omega$ denote $\omega_n=e^{-\frac{2\pi}{n}i}$ and applying the definition
  of $\hat{\rw}$, we have
$$\sum_{k = 0}^{n-1}\gamma(k) \omega^{kx} = \sum_{k=0}^{n-1}\gamma(k)\omega^{ky}.$$
Since $\omega^n = 1$, $\omega$ is a root of the polynomial 
$$P(t) = \sum_{k = 0}^{n-1}\gamma(k) t^{kx\, (mod\, n)} - \sum_{k=0}^{n-1}\gamma(k)t^{ky\, (mod\, n)}.$$
or, by a change of variables
$$P(t) = \sum_{k = 0}^{n-1}\left(\gamma(kx^{-1}) -\gamma(k y^{-1}) \right)t^{k}$$
where the inverses are taken in the field $\mathbb{Z}_n$. (Recall $n$ is prime.)

Since $P$ has $\omega_n$ as a root then $Q_n$, the $n$-th cyclotomic
polynomial, divides $P$. However, $P$ has degree at most $n-1$, so
either $P$ is the zero polynomial or $P$ is equal to a constant times
$Q_n$.

The latter option is impossible, since $P(1) = 0$ and $Q_n(1) =
n$. Thus, $P = 0$, so $\gamma(k x^{-1}) = \gamma(k y^{-1})$ for all
$k$, or equivalently, $\gamma(k x^{-1} y) = \gamma(k)$ for all $k$.
\end{proof}

\begin{lemma}\label{lemma:prime-coupling}
  If $\hat{\gamma}(x) = \hat{\gamma}(y)$ for some $x, y
  \neq 0$, and if $f_1, f_2:\mathbb{Z}_n \rightarrow \{0,1\}$ are such
  that $f_1(k) = f_2(x^{-1}yk)$ for all $k$, then $\gamma$ cannot
  distinguish between $f_1$ and $f_2$.
\end{lemma}

\begin{proof}
  Let $v_1(t)$ be a $\rw$-r.w.\ on $f_1$. Let $v_2(t)$ be a random
  walk defined by $v_2(t) = x^{-1}y v_1(t)$, so that whenever $v_1$
  jumps $k$ then $v_2$ jumps $x^{-1}yk$.

  By Lemma~\ref{lemma:xS=yS} we have that $\rw(x^{-1}yk) = \rw(k)$, so
  the step distribution of $v_2$ is also $\rw$. Furthermore, at time
  $t$, $v_1$ will see $f_1(v_1(t))$ and $v_2$ will see $f_2(v_2(t)) =
  f_2(x^{-1}yv_1(t))$. But for all $k$, $f_1(k) = f_2(x^{-1}y \cdot
  k)$, and so $f_1(v_1(t))=f_2(v_2(t))$ for all $t$.
\end{proof}

To conclude the proof of Theorem~\ref{thm:prime-distinct}, we must
find, for all prime $n > 5$ and all $x \neq y$, $x, y \neq 0$ functions $f_1$ and
$f_2$ such that $f_1$ is not a shift of $f_2$, and $f_1(k) =
f_2(x^{-1}yk)$. We use another argument for the special case $x=0$ or $y = 0$.

\begin{proof}[Proof of Theorem~\ref{thm:prime-distinct}]

  For $x, y \neq 0$, by Lemma~\ref{lemma:prime-coupling} we can choose
  any $f_1$ and $f_2$ such that $f_1(k) = f_2(x^{-1}yk)$ and $f_1$ is
  not a shift of $f_2$; for example, if $x^{-1}y \neq -1$, then
$$f_1(k) = \begin{cases}
  1 & k = 0,1\\
  0 & o.w.
\end{cases}$$
and
$$f_2(k) = \begin{cases}
  1 & k = 0, x^{-1}y\\
  0 & o.w.
\end{cases}$$
satisfy the requirements. For $n=7$ and $x^{-1}y=2$ these two
functions are depicted in Fig.~\ref{fig:seven-walk}.

If $x^{-1}y = -1$, then we can choose any function which is not a
reflection or shift of itself; for example,
$$f_1(k) = \begin{cases} 
1& k = 0, 1, 3 \\
0 & o.w.
\end{cases}$$
and $f_2(k)=f_1(-k)$.

This leaves us with the case $x=0$ or $y=0$; wlog, suppose $x = 0$. We
still have $P(t) \equiv 0$ by the same proof as above, which gives us
\begin{eqnarray*}
0 & \equiv& P(t)\\
& = & \sum_{k = 0}^{n-1}\gamma(k) t^{kx\, (mod\, n)} - \sum_{k=0}^{n-1}\gamma(k)t^{ky\, (mod\, n)}\\
&=& 1- \sum_{k=0}^{n-1}\gamma(k)t^{ky\, (mod\, n)}\\
&=& 1- \sum_{k=0}^{n-1}\gamma(ky^{-1})t^{k}\\
\end{eqnarray*}
where the last equality is possible because $y \neq x=0$. But then we
must have $\gamma(0) = 1$, so the random walk does not move after the
first vertex is chosen. Thus, $\gamma$ cannot distinguish between any
two functions with the same number of ones.
\end{proof}

\section{Walks with non-zero drift and walks with bounded support}
\label{sec:corollaries}

In this section, we will use the main Lemma~\ref{lemma:xS=yS} to prove
Theorems~\ref{cor:drift} and \ref{cor:bounded}.

\subsection{Drift and reconstruction}

Recall from Definition~\ref{def:drift} that the drift of a random walk
on $\Z_n$ with step function $\Gamma$ is $D(\Gamma) = \sum_{k \in
  \Gamma} k$.

\begin{proof}[Proof of Theorem~\ref{cor:drift}]
  We will show that if reconstruction is not possible then the drift
  is zero.
  
  If $f$ cannot be reconstructed then there exist $x\neq y$ with
  $\hat{\rw}(x) = \hat{\rw}(y)$, by Theorem~\ref{thm:distinct-f}. Then
  it may be that $x$ or $y$ is 0; in this case, as we have shown in
  the proof of Theorem~\ref{thm:prime-distinct}, the random walk never
  moves, so the drift is zero. Otherwise $x, y \neq 0$, which by
  Lemma~\ref{lemma:xS=yS} implies $\rw(k) = \rw(x^{-1}yk)$, or
  alternatively $x^{-1}y\Gamma = \Gamma$. But then
\begin{eqnarray*}
  D(\Gamma) &=& \sum_{k \in \Gamma} k\\
  &=& \sum_{k \in \Gamma} x^{-1}yk\\
  &=& xy^{-1} D(\Gamma)\\
\end{eqnarray*}
so $D(\Gamma) = 0$.
\end{proof}

\subsection{Random walks with bounded step distribution}

Next, we consider random walks with rational transition probabilities,
and with zero probability for steps larger than $c$, for some constant
$c$ independent of $n$. We show in Theorem~\ref{cor:bounded} that in
this case, either the random walk is symmetric, or for large enough
prime cycles, the walk is reconstructive. A random walk $\rw$ is
symmetric when $\rw(k) = \rw(-k)$ for all $k$. Symmetric random walks
are not reconstructive, since $\hat{\rw}(k)=\hat{\rw}(-k)$; they
cannot distinguish any function $f(k)$ from its flip $f(-k)$.  We will
need two lemmas in order to prove this theorem.

As we note above, a rational step distribution $\rw$ can be thought of
as a uniform distribution over a multiset $\Gamma=(a_1,\ldots,a_m) \in
\Z^m$. We denote by $\Gamma_n=(k_1,\ldots,k_m) \in \Z_n^m$ the natural
embedding of $\Gamma$ into $\Z_n$: $k_i = a_i \mod n$.

\begin{lemma}
  \label{lemma:Gamma_v}
  Let $\Gamma =(a_1,\ldots,a_m) \in \Z^m$ be a multiset. Let
  $\Gamma_n=(k_1,\ldots,k_m)$ be the natural embedding of $\Gamma$
  into $\Z_n$. Assume further that $1 \in \Gamma$, i.e., $a_i=1$ for
  some $i$.

  Then there exists a positive integer $N = N(\Gamma)$ such that for
  any $n > N$ it holds that if $v\Gamma_n=\Gamma_n$ then $v \in
  \{-1,0,1\}$.
\end{lemma}
\begin{proof}
  \label{lemma:Gamma_n}
  Let $b = \max_i|a_i|$, so that $\Gamma$ is bounded in $[-b,b]$, and
  assume w.l.o.g that $b \in \Gamma$. Let $N=2b^2$, $n>N$ and let
  $v\Gamma_n=\Gamma_n$. We will show that $v \in \{-1,0,1\}$.

  Because $1 \in \Gamma_n$ and $v\Gamma_n = \Gamma_n$ we have $v \in
  \Gamma_n$. Assume by way of contradiction that $v \not \in
  \{-1,0,1\}$. Then $v \in [-b, -2] \cup [2, b]$. Since $b \in
  \Gamma_n$ we also have $vb \in \Gamma_n$, so $vb \in [-b^2, -2b]
  \cup [2b, b^2]$. But this is a contradiction because $n > 2b^2$ so
  both of these intervals have empty intersection with $\Gamma_n$.
\end{proof}
\begin{lemma}
  \label{lemma:symmetric_reconstructive}

  Let $\Gamma_n$ be a multiset characterizing a random walk $\gamma_n$
  over $\Z_n$, such that if $v\Gamma_n=\Gamma_n$ then $v \in
  \{-1,0,1\}$. Then the random walk is either symmetric or
  reconstructive.
\end{lemma}
\begin{proof}

  Suppose the random walk is not reconstructive; i.e. there exist $x \neq
  y$ with $\hat{\rw_n}(x) = \hat{\rw_n}(y)$. We will show that this
  means that the random walk is symmetric.

  If $x$ is 0, then by the argument in the proof of
  theorem~\ref{thm:prime-distinct} we have that $\Gamma = \{0\}$,
  which is symmetric, and we are done.  Otherwise denote $v =
  x^{-1}y$. Then by Lemma~\ref{lemma:xS=yS} we have $\Gamma_n = v
  \Gamma_n$, and by this lemma's hypothesis we have $v \in
  \{-1,0,1\}$.

  Now, $v$ cannot equal $1$ since $v = x^{-1}y$ and $x \neq y$.  If
  $v=-1$ then the random walk is symmetric. Finally, if $v=0$ then
  again $\Gamma = \{0\}$, and again the random walk is symmetric.

\end{proof}

\begin{proof}[Proof of Theorem~\ref{cor:bounded}]
  Let $\Gamma \subset \Z$ correspond to $\rw$ as above (i.e., $\rw$ is
  uniformly distributed over the multiset $\Gamma$). We will show that
  for $n$ large enough if $v\Gamma_n = \Gamma_n$ then $v \in
  \{-1,0,1\}$, which, by Lemma~\ref{lemma:symmetric_reconstructive}
  will show that the random walk is either symmetric or reconstructive.

  For any vector of coefficients $c = (c_1, \ldots, c_m)$, define the
  multiset $c(\Gamma)$ as
  $$c(\Gamma) = c_1\Gamma + \ldots + c_m\Gamma=\{c_1s_1 + \ldots + c_ms_m
  : s_i \in \Gamma\}.$$ Note that if $v\Gamma = \Gamma$, then we
  also have that $v c(\Gamma) = c(\Gamma)$.

  Suppose that the g.c.d.\ (in $\Z$) of the elements of $\Gamma$ is
  one. Then by the Chinese remainder theorem, there exists a $c$ such
  that $1 \in c(\Gamma)$. Hence we can apply Lemma~\ref{lemma:Gamma_v}
  to $c(\Gamma)$ and infer that for $n$ large enough $v c(\Gamma_n) =
  c(\Gamma_n)$ implies $v \in \{-1,0,1\}$. But since $v\Gamma_n =
  \Gamma_n$, implies $v c(\Gamma_n) = c(\Gamma_n)$ then we have that
  for $n$ large enough $v\Gamma_n = \Gamma_n$ implies $v \in
  \{-1,0,1\}$.
  
  Otherwise, let $d$ denote the g.c.d.\ of $\Gamma$.  Since the
  g.c.d.\ of $d^{-1}\Gamma$ is one we can apply the argument of the
  previous case to $d^{-1}\Gamma$ and infer that for $n$ large enough
  $vd^{-1}\Gamma_n = d^{-1}\Gamma_n$ implies $v \in \{-1,0,1\}$. But
  since again $v\Gamma_n = \Gamma_n$ implies $vd^{-1}\Gamma_n =
  d^{-1}\Gamma_n$, then again we have that for $n$ large enough
  $v\Gamma_n = \Gamma_n$ implies $v \in \{-1,0,1\}$.

\end{proof}

\section{Tightness of Theorem~\ref{thm:prime-distinct}}
\label{sec:tightness}

Theorem~\ref{thm:prime-distinct} states that distinctness of Fourier coefficients is not just sufficient but necessary for reconstructibility of a random walk $\rw$ on a cycle of length $n$, if the following conditions hold:
\begin{enumerate}
\item $\rw(k) \in \Q$ for all $k$,
\item $n> 5$,
\item $n$ is prime.
\end{enumerate}

To see that condition 2 is tight, note that the simple random walk has
$\hat{rw}(k) = \hat{rw}(-k)$ for all $k$, so for $n > 2$, the simple
random walk does not have distinct Fourier coefficients. However,
Howard~\cite{Howard:96} shows that the simple random walk suffices to
reconstruct any scenery up to a flip, and for $n \leq 5$, reconstruction up
to a flip is the same as reconstruction up to a shift, since all
sceneries are symmetric. So for $n = 3,4,5$, the simple random walk is
reconstructive despite having non-distinct Fourier coefficients. For $n =
2$, consider the random walk that simply stays in place.

The following theorem shows that condition 1 is also tight.

\begin{theorem}\label{thm:counterexample}
There exists a reconstructive random walk $\rw$ on $\mathbb{Z}_7$ such that $\hat{\rw}(3) =\hat{\rw}(-3)$.
\end{theorem}

\begin{proof}
%\begin{counterexample}
%\label{ce:delta}
Let $\delta = \frac{\cos\left( 6\pi /7\right) + 0.5}{2
  \cos\left(6\pi/7 \right) - 1}$, $\rw(1) = \half + \delta$ and
$\rw(2) = \half - \delta$. To show that $\hat{\rw}(3) = \hat{\rw}(-3)$ is a simple calculation, so we will only show that $\rw$ can distinguish between any two sceneries.\\

In fact, any random walk with non-zero support exactly on $\{1,2\}$ can distinguish between any two sceneries when $n = 7$. Any r.w. can determine the number of ones in a scenery, so we must show that $\rw$ can distinguish between non-equivalent sceneries for each number of ones. 
\begin{itemize}
\item There is only one scenery with 0 ones and up to a shift, there is only one scenery with 1 one.
\item For 2 ones, $\rw$ must distinguish among $(1,1,0,0,0,0,0), (1,0,1,0,0,0,0), (1,0,0,1,0,0,0)$. (The others are equivalent to one of these up to a shift.) The third function is the only one for which there will never be two consecutive ones. The first function is the only one for which there will be two consecutive ones, but the substring $(1,0,1)$ never appears. The third function is the only one for which both substrings $(1,1)$ and $(1,0,1)$ will appear.
\item For 3 ones, $\rw$ must distinguish among: 
(a) $(1,1,1,0,0,0,0)$, 
(b) $(1,1,0,1,0,0,0)$, 
(c) $(1,1,0,0,1,0,0$), 
(d) $(1,1,0,0,0,1,0)$, and 
(e) $(1,0,1,0,1,0,0)$. 
(c) is the only one that will never have three consecutive ones. (e) is the only that will ever have five consecutive zeros. Among (a), (b), and (d), (a) is the only such that (1,0,1) will never occur, and while (1,1,0,1,0,0) will occur for (b), it won't for (d).
\item For 4,5,6, and 7 ones, repeat the previous arguments with the roles of 0 and 1 reversed.
\end{itemize}
\end{proof}

\section{Extensions}
\label{sec:extensions}

\subsection{Extension to prime regular tori}

\begin{theorem}\label{thm:regular_torus}
  Let $\rw$ be the step distribution of a random walk $v(t)$ on $\Z_n^d$, for
  $n$ prime and larger than five, and let $\rw(k)$ be rational for all
  $k$.  Then $v(t)$ is reconstructive {\em only if} the Fourier
  coefficients $\{\hat{\rw}(x)\}_{x \in \Z_n}$ are distinct.
\end{theorem}
\begin{proof}
  Suppose that $\hat{\gamma}(x) = \hat{\gamma}(y)$ for some $x \neq
  y$, so that the Fourier coefficients aren't distinct.  Again letting
  $\omega$ denote $\omega_n=e^{-\frac{2\pi}{n}i}$, we have
  \begin{align*}
    \sum_{k \in \Z_n^d}\gamma(k) \omega^{k \cdot x} =
    \sum_{k \in \Z_n^d}\gamma(k)\omega^{k \cdot y}.
  \end{align*}
  Note that here $k \cdot y$ is the natural dot product over
  $\Z_n^d$. If we join terms with equal powers of $\omega$ then 
  \begin{align*}
    \sum_{\ell \in \Z_n}\left(\sum_{\substack{k\;\mathrm{s.t.} \\ k \cdot x
        = \ell}}\gamma(k)\right) \omega^\ell =
    \sum_{\ell \in \Z_n}\left(\sum_{\substack{k\;\mathrm{s.t.} \\ k \cdot y
        = \ell}}\gamma(k)\right) \omega^\ell 
  \end{align*}
  and by the same argument as in the proof of Lemma~\ref{lemma:xS=yS} we have that,
  for all $\ell \in \Z_n$,
  \begin{align}
    \label{eq:same-dist-torus}
    \sum_{\substack{k\;\mathrm{ s.t. }\\  k \cdot x = \ell}}\gamma(k) =
    \sum_{\substack{k\;\mathrm{ s.t. }\\  k \cdot y = \ell}}\gamma(k)
  \end{align}
  I.e., if $k$ is distributed according to $\rw$ then $k \cdot x$ and
  $k \cdot y$ have the same distribution. Hence, if we denote
  $u_x(t)=v(t) \cdot x$ and $u_y(t) = v(t) \cdot y$ then we also have
  identical distributions, under $\rw$, of $u_x(t)$ and $u_y(t)$.
  
  Fix $g:\Z_n \to \{0,1\}$, and let $f_x(k) = g(x \cdot k)$, and
  $f_y(k) = g(y \cdot k)$.  Then $f_x(v(t)) = g(u_x(t))$, $f_y(v(t)) =
  g(u_y(t))$, and the distributions of $\{f_x(v(t)\}$ and
  $\{f_y(v(t))\}$ are identical, by the same coupling argument used in
  Lemma~\ref{lemma:prime-coupling} above. Hence $f_x$ and $f_y$ can't
  be distinguished.

  It remains to show that there exists a $g$ such that $f_x$ and $f_y$
  differ by more than a shift.  We consider two cases.
  \begin{enumerate}
  \item Let $x$ be a multiple of $y$, so that $x=\ell y$ for some $
    \ell \in \Z_n$, $\ell \neq 1$. Then the problem is essentially
    reduced to the one dimensional case of
    Theorem~\ref{thm:prime-distinct}, with the random walk projected
    on $y$. We can set
    $$g(k) = \begin{cases} 
      1& k = 0, 1 \\
      0 & o.w.
    \end{cases}$$
    when $\ell \neq -1$ and
    $$g(k) = \begin{cases} 
      1& k = 0, 1, 3 \\
      0 & o.w.
    \end{cases}$$
    when $\ell = -1$.
    
    Assume by way of contradiction that $f_x$ and $f_y$ differ by a
    shift, so that there exists a $k_0$ such that $f_x(k)=f_y(k +
    k_0)$ for all $k$. Since $$f_x(k) = g(x \cdot k) = g(\ell y \cdot
    k)$$ and $$f_y(k+k_0) = g(y \cdot (k+k_0)) = g(y \cdot k+y \cdot
    k_0)$$ for all $k$, then $$g(\ell m)=g(m+m_0)$$ for some $m_0 \in
    \Z_n$ and all $m \in \Z_n$.  That is, $g(m)$ is a shift of $g(\ell
    m)$. It is easy to verify that this is not possible in the case
    $\ell \neq -1$ nor in the case $\ell = -1$.

  \item Otherwise $x$ is not a multiple of $y$. Hence they are
    linearly independent.

    We here set
    $$g(k) = \begin{cases} 
      1& k = 0 \\
      0 & o.w.
    \end{cases}$$ One can then view $f_x$ ($f_y$), as the indicator
    function of the set elements of $\Z_n^d$ orthogonal to $x$
    ($y$). Denote these linear subspaces as $U_x$ and $U_y$.
    
    Assume by way of contradiction that $f_x$ and $f_y$ differ by a
    shift, so that there exists a $k_0$ such that $f_x(k)=f_y(k +
    k_0)$ for all $k$. Then $U_x = U_y + k_0$. Since $0$ is an element
    of both $U_x$ and $U_y$ then it follows that $k_0$ is also an
    element of both. But, being a linear space, $U_y$ is closed under
    addition, and so $U_x = U_y+k_0 = U_y$. However, since $x$ and $y$
    are linearly independent then their orthogonal spaces must be
    distinct, and we've reached a contradiction.
  \end{enumerate}
  
\end{proof}
\subsection{Extension to products of prime tori}

Recall that any abelian group $H$ can be decomposed into
$H=\Z_{n_1}^{d_1}\times \cdots \times \Z_{n_m}^{d_m}$, where the $n_i$
are distinct powers of primes. We call $H$ ``square free'' when all
the $n_i$'s are in fact primes. In this case too (for primes $>5$) we
show in Theorem~\ref{thm:general} that the distinctness of Fourier
coefficients is a tight condition for reconstructibility.

We can assume w.l.o.g.\ that the $p_i$ are all distinct. We will prove
two lemmas before proving the theorem.

Recall that $\Q(\omega_{p_1}, \ldots, \omega_{p_{j-1}})$ is the field
extension of $\Q$ by $\omega_{p_1}, \ldots, \omega_{p_{j-1}}$. The following is a straightforward algebraic claim and we provide the proof for completeness.
\begin{lemma}
  \label{lemma:min_poly}

  For all $j$, the minimal polynomial $Q_j$ of $\omega_{p_j}$ over
  $\Q(\omega_{p_1}, \ldots, \omega_{p_{j-1}})$ is
  $$Q_j(t) = \sum_{i = 0}^{p_j - 1}t^i.$$
\end{lemma}

\begin{proof}
  Let $F_j = \Q(\omega_{p_1}, \ldots, \omega_{p_j})$, let $F_{j-1} =
  \Q(\omega_{p_1}, \ldots, \omega_{p_{j-1}})$, and let $\zeta
  =\omega_{p_1}\cdots \omega_{p_j}$. The field $F_j$ contains $
  \zeta$ and $\Q$, so $\Q(\zeta) \subset F_j$. Since each
  $\omega_{p_i}$ is a power of $\zeta$, $\Q$ and $\omega_{p_1}, \ldots
  , \omega_{p_j}$ are all in $\Q(\zeta)$; thus, $F_j \subset
  \Q(\zeta)$. This allows us to conclude that $F_j = \Q(\zeta)$.
  
  The degree of $\zeta$ over $\Q$ is $(p_1-1)\cdots (p_j-1)$, since
  $\zeta$ is a primitive $n$-th root of unity and $(p_1-1)\cdots
  (p_j-1)$ is the number of integers coprime to $n$. Thus, the field
  extension $\Q \subset F_j$ has degree $[F_j:\Q] = [\Q(\zeta):\Q] =
  (p_1 - 1) \cdots (p_j - 1)$. By a similar argument, $[F_{j-1}:\Q] =
  (p_1-1)\cdots (p_{j-1}-1)$. So degree of the field extension
  $F_{j-1} \subset F_j$ is
  $$[F_j:F_{j-1}] = [F_j : \Q] / [F_{j-1}:\Q] = p_j - 1.$$ 

  Since the degree of the minimal polynomial of $\omega_{p_j}$ over
  $F_{j-1}$ is the same as the degree of $F_j$ over $F_{j-1}$, we can
  conclude that the degree of this minimal polynomial is $p_j -
  1$. Since this is also the degree of $\omega_{p_j}$ over $\Q$, the
  minimal polynomial of $\omega_{p_j}$ over $F_{j-1}$ is the same as
  the minimal polynomial of $\omega_{p_j}$ over $\Q$. Thus, the
  minimal polynomial $Q_j$ of $\omega_{p_j}$ over $F_{j-1}$ is
  $$Q_j(t) = \sum_{i = 0}^{p_j - 1}t^i,$$
  as desired.
\end{proof}

\begin{lemma}\label{lemma:general_coupling}
Suppose $\hat{\rw}(x_1, \ldots x_m) = \hat{\rw}(y_1, \ldots y_m)$ for 
$(x_1, \ldots, x_m), (y_1, \ldots, y_m) \in \mathbb{Z}_{p_1}^{d_1} \times
  \cdots \times \mathbb{Z}_{p_m}^{d_m}$. Then 
for all $(\ell_1, \ldots , \ell_m) \in \mathbb{Z}_{p_1} \times
  \cdots \times \mathbb{Z}_{p_m}$, $$\sum_{\substack{(k_1,
      \ldots, k_m) \\ k_i \cdot x_i = \ell_i \, \forall i}} \rw(k_1,
  \ldots, k_m) = \sum_{\substack{(k_1, \ldots, k_m) \\ k_i \cdot y_i =
      \ell_i\, \forall i}} \rw(k_1, \ldots, k_m).$$
\end{lemma}

\begin{proof}
We will prove by induction on $j$ that the following statement holds: for all $(\ell_{j+1}, \ldots , \ell_m) \in \mathbb{Z}_{p_{j+1}} \times \cdots \times \mathbb{Z}_{p_m}$, 

$$\sum_{\substack{(k_1, \ldots, k_m)\\k_i \cdot x_i = \ell_i \, \forall i>j}} \rw(k_1, \ldots, k_m) \omega_{p_1}^{k_1 \cdot x_1} \cdots \omega_{p_{j }}^{k_{j } \cdot x_{j}} 
= \sum_{\substack{(k_1, \ldots, k_m) \\ k_i \cdot y_i = \ell_i \, \forall i>j}} \rw(k_1, \ldots, k_m) \omega_{p_1}^{k_1 \cdot y_1} \cdots \omega_{p_{j }}^{k_{j} \cdot y_{j}}.$$
The base case $j = m$ is just the statement $\hat{\rw}(x_1, \ldots, x_m) = \hat{\rw}(y_1, \ldots, y_m)$. Now suppose we know that for all $(\ell_{j+1}, \ldots, \ell_m)$

$$\sum_{\substack{(k_1, \ldots, k_m)\\k_i \cdot x_i = \ell_i \, \forall i>j}} \rw(k_1, \ldots, k_m) \omega_{p_1}^{k_1 \cdot x_1} \cdots \omega_{p_{j}}^{k_{j} \cdot x_{j}} 
= \sum_{\substack{(k_1, \ldots, k_m) \\ k_i \cdot y_i = \ell_i \, \forall i>j}} \rw(k_1, \ldots, k_m) \omega_{p_1}^{k_1 \cdot y_1} \cdots \omega_{p_{j }}^{k_{j} \cdot y_{j}}.$$
Then $\omega_{p_{j}}$ is a root of the polynomial
$$P(t) = \sum_{\substack{(k_1, \ldots, k_m)\\k_i \cdot x_i = \ell_i \, \forall i>j}} \rw(k_1, \ldots, k_m) \omega_{p_1}^{k_1 \cdot x_1} \cdots \omega_{p_{j-1}}^{k_{j-1} \cdot x_{j-1}} t^{k_j \cdot x_j} 
- \sum_{\substack{(k_1, \ldots, k_m) \\ k_i \cdot y_i = \ell_i \, \forall i>j}} \rw(k_1, \ldots, k_m) \omega_{p_1}^{k_1 \cdot y_1} \cdots \omega_{p_{j-1 }}^{k_{j-1} \cdot y_{j-1}} t^{k_j \cdot y_j}.$$
But as in the proof of the previous theorem, Lemma~\ref{lemma:min_poly} tells us that $P(t) = Q_j(t)$ or $P(t) \equiv 0$, and so $P(1) = 0 \neq p_j = Q_j(1)$ tells us that $P(t) \equiv 0$. Thus, the coefficient of $t^{\ell_j}$ must be zero for all choices of $\ell_j$, establishing the desired statement for $j-1$. When $j= 0$, this gives us the statement of the lemma.
\end{proof}

\begin{proof} [Proof of Theorem~\ref{thm:general}]
Suppose $\hat{\rw}(x_1, \ldots, x_m) = \hat{\rw}(y_1, \ldots, y_m)$ for $(x_1, \ldots, x_m) \neq (y_1, \ldots, y_m)$. Lemma~\ref{lemma:general_coupling} allows us to construct the following coupling. Let $v_1(1)$ be uniform and $v_1(t) - v_1(t-1)$ be drawn according to $\rw$. Let $v_2(t)$ be drawn according $\rw$, coupled so that $v_2(t)_i \cdot y_i = v_1(t)_i \cdot x_i$ for all $i$. By Lemma~\ref{lemma:general_coupling}, this induces the correct distribution on $v_2(t)$. Now, choose an index $j$ such that $x_j \neq y_j$, and let $f_1(k) = g(x_i \cdot k_i)$ and $f_2(k) = g(y_i \cdot k_i)$. The rest of the proof follows as in the proof of Theorem~\ref{thm:regular_torus}.
\end{proof}

\section{Open Problems}
\label{sec:openproblems}

This work leaves open many interesting questions; we sketch some here.

\begin{itemize}

\item Here, we give an equivalent condition for reconstructivity when (a) the random walk is rational, and (b) the underlying graph corresponds to a group of the form $\mathbb{Z}_{p_1}^{d_1} \times \cdots \times \mathbb{Z}_{p_k}^{d_k}$ for distinct primes $p_1, \ldots, p_k >5$. We know that this condition is not necessary when the random walk is irrational, or when $p_i = 3$ or 5 for some $i$. But is this condition necessary when the random walk is rational and the graph is a cycle of size, say, 27? What (if any) is an equivalent condition when the random walk is irrational?

\item The techniques used in this paper do not extend directly to non-abelian groups, because the Fourier transform on a non-abelian group is very different from the Fourier transform on an abelian group. Nevertheless, it is possible that a sufficient condition for reconstructivity similar to the one given in Theorem~\ref{thm:distinct-f} exists for non-abelian groups; proving such a condition is an interesting challenge.

\item We focus only on characterizing which random walks are
  reconstructive. But there are many open questions about the
  equivalence class structure of non-reconstructive walks; for
  example:
  \begin{itemize}
  \item For a given r.w., how many non-minimal equivalence classes are
    there, and of what size?
  \item If the random walk has bounded range, are the equivalence
    classes of bounded size?
  \item Are sceneries in the same equivalence class far from each
    other (i.e. one cannot be obtained from the other by a small
    number of changes)?
  \end{itemize}

\item One could also ask more practical questions, such as: how many
  steps are necessary to distinguish between two sceneries? This could
  be related, for example, to the mixing time of the random walk
  and/or its Fourier coefficients. A harder problem is: how many steps
  are necessary to reconstruct a scenery?  Or: find an efficient
  algorithm to reconstruct a scenery. Such problems have been tackled
  before; see for example~\cite{Matzinger:2003}.

\item An interesting question is the reconstruction of the step
  distribution, rather than the scenery. Which {\em known} function
  $f(k)$ would minimize the number of observations needed to
  reconstruct the {\em unknown} step distribution $\rw$ of a random
  walk $v(t)$?

\item The entropy of $\{f(v_t)\}$, which depends both on the random
  walk and on the scenery, is a natural quantity to explore. Specifically, we 
  can consider the average entropy per observation as the number of observations 
  goes to infinity.  In general, two sceneries that induce the same entropy for a 
  given random walk do not have to be the same; for example,
  $(1,0,1,0,\ldots )$ and $(1,1,1,\ldots)$ both have entropy going to zero for
  the simple random walk. (The first random walk will have one bit of entropy 
  reflecting the starting bit, but as the number of observations grows, this 
  becomes negligible.) Consider two functions equivalent under a
  given random walk when they induce the same entropy. What do these
  equivalence classes look like? Are there random walks for which the
  equivalence classes are minimal? What is the entropy of a random
  scenery under the simple random walk? Does it decrease if the walk
  is biased to one side?
\end{itemize}

\section{Acknowledgments} The authors would like to thank Itai Benjamini, Elchanan Mossel, Yakir Reshef, and Ofer Zeitouni for helpful conversations.

\pagebreak
 
\appendix

\section{Matzinger and Lember's theorem for abelian groups}
We provide in this Appendix the proof of the following theorem, which
is a straightforward generalization of Matzinger and Lember's
Theorem~\ref{thm:distinct-f} to abelian groups, and follows their
proof scheme closely.

\begin{theorem}[Matzinger and Lember]
  \label{thm:ML-abelian}
  Let $\rw$ be the step distribution of a random walk on a finite abelian
  group $H$. Let $\hat{\rw}$ be the Fourier Transform of $\rw$. Then
  $f$ can be reconstructed up to shifts if the Fourier coefficients
  $\{\hat{\rw}(x)\}_{x \in \Z_n}$ are distinct.
\end{theorem}

Let $H=\Z_{n_1}^{d_1}\cdots\Z_{n_m}^{d_m}$ be an abelian group, with
the $n_i$'s distinct powers of primes. Let $n$ be the order of $H$.
\subsection{Autocorrelation}
Given a function $f$ on $H$, the autocorrelation $a_f(\ell)$ is
defined as the inner product of $f$ with $f$ transformed by a
$k$-shift:
\begin{equation}
  \label{eq:autocorr}
  a_f(\ell)=\sum_{k \in H}f(k) \cdot f(k+\ell).
\end{equation}
When $f:H \to \{0,1\}$ is a labeling of the vertices of the random
walk then we refer to $a_f$ as the {\em spatial autocorrelation}.

It follows directly from the convolution theorem that $\hat{a}_f$, the
Fourier transform of $a_f$, is equal to the absolute value squared of
$\hat{f}$:
\begin{equation}
  \hat{a}_f(x)=|\hat{f}(x)|^2.
\end{equation}

We consider a r.w.\ $v(t)$ over $H$ with vertices labeled by $f: H \to
\{0,1\}$. As above, we let $v(1)$ be chosen uniformly at random, and
set $\rw(k) = \P{v(t+1)-v(t)=k}$.

We define the {\em temporal autocorrelation} $b_f$ in the same spirit:
\begin{equation}
  \label{eq:temporal-autocorr}
  b_f(\ell)=\E{f(v(T)) \cdot f(v(T+\ell))}.
\end{equation}
Note that since the random walk is stationary then the choice of $T$
is immaterial and $b_f$ is indeed well defined.

There is a linear relation between the spatial and temporal
autocorrelations:
\begin{theorem}
  \label{thm:linear1}
  \begin{equation}
    \label{thm:bf-af}
    b_f(\ell) = \frac{1}{n}\sum_{x \in H}\hat{\rw}(x)^l\hat{a}_f(x)    
  \end{equation}
\end{theorem}
\begin{proof}
  We denote by $\rw^{(\ell)}$ the function $\rw$ convoluted with
  itself $\ell$ times, over $H$.  It is easy to convince oneself
  that $\rw^{(\ell)}(k)$ is the distribution of the distance traveled
  by the particle in $\ell$ time periods, so that $\rw^{(\ell)}(k) =
  \P{v(t+\ell)-v(t)=k}$.

  Now, by definition we have that 
\begin{align*}
  b_f(\ell)&=\E{f(v(T)) \cdot f(v(T+\ell))}.\\
\end{align*}
Since the random walk starts at a uniformly chosen vertex then
\begin{align*}
  b_f(\ell)&= \frac{1}{n} \sum_k \mathbb{E}\left[ f(v(T)) \cdot
    f(v(T+\ell))\, \vline \, v(T) = k\right].\\
\end{align*}
Conditioning on the step taken at time $T$ we get that 
\begin{align*}
  b_f(\ell)&= \frac{1}{n} \sum_{k,x} f(k) \cdot \rw^{(\ell)}(x)\cdot f(k+x).\\
\end{align*}
Finally, changing the order of summation and substituting in $a_f(x)$
we get that 
\begin{align*}
  b_f(\ell)&= \frac{1}{n} \sum_x \rw^{(\ell)}(x) \cdot a_f(x).
\end{align*}

Now, the expression $\sum_x \rw^{(\ell)}(x) \cdot a_f(x)$ can be
viewed as a dot product of two vectors in the vector space of
functions from $H$ to $\C$. Since the Fourier transform is an
orthogonal transformation of this space we can replace these two
functions by their Fourier transforms. By the convolution theorem we
have that $\widehat{\rw^{(\ell)}} = \hat{\rw}^\ell$, and so
\begin{align*}
  b_f(\ell) 
  &= \frac{1}{n}\left(\rw^{(\ell)}, a_f\right)\\
  &= \frac{1}{n}\left(\hat{\rw}^\ell, \hat{a}_f\right)\\
  &= \frac{1}{n}\sum_{x \in \Z_n}\hat{\rw}(x)^l\hat{a}_f(x).    
\end{align*}

\end{proof}

It follows from Theorem~\ref{thm:linear1} that the spatial
autocorrelation $a_f$ can be calculated from the temporal
autocorrelation $b_f$ whenever the Fourier coefficients
$\{\hat{\rw}(x)\}$ are distinct: the linear transformation mapping
$a_f$ to $b_f$ is a Vandermonde matrix which is invertible precisely
when the values of $\hat{\rw}$ are distinct.

By its definition, $b_f$ is determined by the distribution of
$\{f(v(t))\}_{t=1}^\infty$. Hence it follows that:
\begin{corollary}
  When the Fourier coefficients of $\rw$ are distinct then $b_f$
  uniquely determines $a_f$.
\end{corollary}

However, it is possible for $f_1$ not to be a shift of $f_2$, but for
$a_{f_1}$ to equal $a_{f_2}$. Thus, the above argument does not
suffice to prove Theorem~\ref{thm:distinct-f}.

\subsection{Bispectrum and beyond}
A generalization of autocorrelation is the {\em bispectrum} (see,
e.g.,~\cite{Mendel:91}):

\begin{equation}
  \label{eq:bispectrum}
  A_f(\ell_1, \ell_2) = \sum_{k \in H}f(k) \cdot f(k+\ell_1) \cdot
  f(k+\ell_1+\ell_2).
\end{equation}

We generalize further to what we call the multispectrum:
\begin{equation}
  \label{eq:multispectrum}
  A_f(\ell_1, \ldots, \ell_{n-1}) = \sum_{k \in H}f(k) \cdot
  f(k+\ell_1) \cdots f(k + \ell_1 + \cdots +\ell_{n-1}).
\end{equation}

As above, we call $A_f$ the spatial multispectrum. The relation
between the Fourier transform of $A_f$ and $f$ is the following. Note
that we transform $A_f$ over $H^{n-1}$ (where $n=|H|$):
\begin{equation}
  \label{Af-transform}
  \hat{A}_f(x_1, \ldots, x_{n-1}) = \overline{\hat{f}(x_1)} \cdot \hat{f}(x_{n-1})
  \cdot \prod_{i=1}^{n-2}\hat{f}(x_i - x_{i+1}),
\end{equation}
where $\overline{z}$ denotes the complex conjugate of $z$.

We define as follows the temporal multispectrum:
\begin{equation}
  \label{eq:temporal-multi}
  B_f(\ell_1, \ldots, \ell_{n-1}) = \E{f(v(T)) \cdot f(v(T + \ell_1))
    \cdots f(v(T + \ell_1 + \cdots + \ell_{n-1}))}.
\end{equation}
Note that, like $b_f$, $B_f$ is determined by the distribution of
$\{f(v(t))\}_{t=1}^\infty$.

\subsection{Distinct Fourier Coefficients of $\rw$ imply reconstruction}

Whereas $a_f$ did not suffice to recover $f$, $A_f$ suffices:
\begin{lemma}
  Suppose $A_{f_1} = A_{f_2}$. Then $f_1$ is a shift of $f_2$.
\end{lemma}
\begin{proof}

  First, note that $A_{f}(\ell_1, \ldots \ell_n)>0$ iff there exists a
  $k$ such that $f(k) = f(k+\ell_1) = \cdots = f(k + \ell_1 + \cdots
  +\ell_n) = 1$. Now, let $\hat{\ell}(f) := (\hat{\ell}_1(f), \ldots ,
  \hat{\ell}_n(f))$ be the lexicographically smallest $n$-tuple
  satisfying $A_f\left(\hat{\ell}_1(f), \ldots ,
    \hat{\ell}_n(f)\right) > 0.$ Note that since $A_{f_1} = A_{f_2}$,
  we must also have $\hat{\ell}(f_1) = \hat{\ell}(f_2).$
  
  But $\hat{\ell}(f)$ determines $f$. Indeed, let $i$ be the largest
  index such that $\ell_1 + \cdots + \ell_i \leq n$. Then there is a
  $k$ such that $k$, $k+\ell_1$, $k+\ell_1 + \ell_2$, \ldots , $k +
  \ell_1 + \cdots + \ell_i$ are exactly the indices at which $f$ is 1:
  if any of these were indices at which $f$ was 0 then
  $A_f(\hat{\ell}(f))$ would be zero, and if there were any more
  indices at which $f$ was 1 then $\hat{\ell}$ would not be
  lexicographically the smallest.

  Thus, $f$ can be recovered uniquely up to a shift from $A_f$ by a
  simple algorithm: find $\hat{\ell}(f)$, then read off the ones from
  the first several entries of $\hat{\ell}(f)$ as above, setting $k$
  to be 0 w.l.o.g.

  % Let $g_1, \ldots, g_{2^n}$ be the functions $H$ to
%  $\{0,1\}$, and let $\vec{\ell}_i$ be an $n$-tuple that describes the
%  jumps of a walk that visits all the locations where $g_i(k)=1$;
%  there exists a $k_i$ such that, for all $k$, $g_i(k) = 1$ iff
%  $k_i+\sum_{j=1}^p \ell_j = k$ for some $p$. Clearly, if
%  $\sum_kg_i(k) > \sum_kg_j(k)$ then $A_{g_j}(\ell_i) =0$, and if
%  $\sum_kg_i(k) = \sum_kg_j(k)$, then $A_{g_j}(\ell_i) > 0$ iff $g_j$
%  is a shift of $g_i$. In this case the shift will be $k_i-k_j$.
%
%  Now, suppose $A_{g_i} = A_{g_j}$. If $\sum_kg_i(k) =
%  \sum_kg_j(k)$ then $A_{g_i}(\vec{\ell}_j) = A_{g_j}(\vec{\ell}_j)
%  > 0$ implies that $g_i$ is a shift of $g_j$. Otherwise suppose
%  w.l.o.g.\ that $\sum_kg_i(k) > \sum_kg_j(k)$, then $A_{g_j}(\vec{\ell}_i) =0
%  \neq A_{g_i}(\vec{\ell}_i)$. Thus $A_{g_i} = A_{g_j}$ if and only if
%  $g_i$ is a shift of $g_j$.
%  
\end{proof}

The following analogue of Theorem~\ref{thm:linear1} will be useful:
\begin{lemma}
  \begin{equation}
    \hat{B}_f(\ell_1, \ldots, \ell_{n-1}) = \frac{1}{n}\sum_{x_1, \ldots, x_{n-1}}\hat{\rw}(x_1)^{\ell_1} \cdots \hat{\rw}(x_{n-1})^{\ell_n} \cdot \hat{A}_f(x_1, \ldots, x_{n-1}).  
  \end{equation}
\end{lemma}
\begin{proof}
  The steps of this proof closely follow those of the proof of
  Theorem~\ref{thm:linear1}.
\begin{align*}
 B_f(\ell_1, \ldots, \ell_{n-1}) 
 &= \E{f(v(T)) \cdot f(v(T + \ell_1)) \cdots f(v(T + \ell_1 + \cdots + \ell_{n-1}))}\\
 &= \frac{1}{n}\sum_{k,x_1, \ldots, x_{n-1}} f(k) \cdot \rw^{(\ell_1)}(x_1)f(k+x_1)  \cdots \rw^{(\ell_{n-1})}(x_{n-1})f(k+x_1 + \ldots + x_{n-1})\\
 &= \frac{1}{n}\sum_{x_1, \ldots, x_{n-1}} \rw^{(\ell_1)}(x_1)\cdots \rw^{(\ell_{n-1})}(x_{n-1}) \sum_k f(k)f(k+x_1)\cdots f(k+x_1 + \ldots + x_{n-1})\\
 &= \frac{1}{n}\sum_{x_1, \ldots, x_{n-1}} \rw^{(\ell_1)}(x_1)\cdots \rw^{(\ell_{n-1})}(x_{n-1}) A_f(x_1, \ldots x_{n-1})
\end{align*}
Applying the Fourier transform as in Theorem~\ref{thm:linear1}, we obtain the lemma.
\end{proof}

We now have what we need to complete the proof of Theorem~\ref{thm:distinct-f}:

\begin{proof}[Proof of Theorem~\ref{thm:ML-abelian}]
We need only show that $B_{f_1} = B_{f_2}$ implies that $A_{f_1} = A_{f_2}$. But this follows because the matrix $\Gamma$ mapping $A_f$ to $B_f$ is a tensor of $M$ with itself $n-1$ times, where $M$ is the matrix mapping $a_f$ to $b_f$. When the values of $\rw$ are distinct, $M$ is invertible, so $\Gamma$ is also invertible, giving us the desired result.
\end{proof}

\bibliographystyle{abbrv} \bibliography{all}

\end{document}